\documentclass{amsart}
 \newtheorem{thm}{Theorem}[section]
 \newtheorem{cor}[thm]{Corollary}
 \newtheorem{lem}[thm]{Lemma}
 
 \theoremstyle{definition}
 
 \theoremstyle{remark}
 
 \newtheorem*{ex}{Example}
 \numberwithin{equation}{section}

\begin{document}
\title[Semisimple Hopf algebras]{Structure of semisimple Hopf algebras of dimension $p^2q^2$}

\author[J. Dong]{Jingcheng Dong}
\address{College of Engineering, Nanjing Agricultural University, Nanjing
210031, Jiangsu, People's Republic of China; Department of
Mathematics, Yangzhou University, Yangzhou 225002, Jiangsu, People's
Republic of China}

\email{dongjc@njau.edu.cn}

\subjclass[2000]{16W30}

\keywords{semisimple Hopf algebra, semisolvability, character,
biproduct}

\begin{abstract}
Let $p,q$ be prime numbers with $p^4<q$, and $k$ an algebraically
closed field of characteristic $0$. We show that semisimple Hopf
algebras of dimension $p^2q^2$ can be constructed either from group
algebras and their duals by means of extensions, or from Radford
biproduct $R\# kG$, where $kG$ is the group algebra of group $G$ of
order $p^2$, $R$ is a semisimple Yetter-Drinfeld Hopf algebra in
${}^{kG}_{kG}\mathcal{YD}$ of dimension $q^2$. As an application,
the special case that the structure of semisimple Hopf algebras of
dimension $4q^2$ is given.
\end{abstract}
\maketitle



\section{Introduction}\label{sec1}
Throughout this paper, we will work over an algebraically closed
field $k$ of characteristic $0$.

The problem of classifying all Hopf algebras of dimension $d$, where
$d$ factorizes in a simple way, attracts many mathematicians'
interest. It is also a question posed by Andruskiewitsch
\cite[Question 6.2]{Andruskiewitsch}. As a pioneer, Zhu \cite{Zhu}
proved that a Hopf algebra of prime dimension over $k$ is a group
algebra. Several years later, a series of papers
\cite{Etingof,Gelaki,Masuoka,Masuoka2} proved that semisimple Hopf
algebras of dimension $p^2$ or $pq$ over $k$ are trivial, where
$p,q$ are distinct prime numbers. That is, they are isomorphic to a
group algebra or to a dual group algebra. Quite recently, Etingof et
al \cite{Etingof2} completed the classification of semisimple Hopf
algebras of dimension $pq^2$ and $pqr$, where $p,q,r$ are distinct
prime numbers. The results in \cite{Etingof2} showed that all these
Hopf algebras can be constructed from group algebras and their duals
by means of extensions.

In this paper, we study the structure of semisimple Hopf algebras of
dimension $p^2q^2$, where $p,q$ are prime numbers with $p^4<q$. As
an application, we also study the structure of semisimple Hopf
algebras of dimension $4q^2$, where $q$ is a prime number.

The paper is organized as follows. In Section \ref{sec2}, we recall
the definitions and basic properties of semisolvability, characters
and Radford's biproducts, respectively.

In Section \ref{sec3}, we study the structure of semisimple Hopf
algebras of dimension $p^2q^2$, where $p,q$ are prime numbers with
$p^4<q$. By checking the order of $G(H^*)$, we prove that if
$|G(H^*)|=p,pq,q^2$ or $pq^2$ then $H$ is not simple as a Hopf
algebra and is semisolvable, in the sense of \cite{Montgomery}; if
$|G(H^*)|=p^2$ or $p^2q$ then $H$ is either semisolvable or
isomorphic to a Radford's biproduct $R\# kG$, where $kG$ is the
group algebra of group $G$ of order $p^2$, $R$ is a semisimple
Yetter-Drinfeld Hopf algebra in ${}^{kG}_{kG}\mathcal{YD}$ of
dimension $q^2$. The possibility that $|G(H^*)|=1$ and $q$ can be
discarded. In particular, we prove that if $p$ does not divide $q-1$
and $q+1$, then $H$ is necessarily semisolvable.

In Section \ref{sec4}, we study the structure of semisimple Hopf
algebras of dimension $4q^2$, where $q$ is a prime number. In view
of the results in Section \ref{sec3}, we discuss the cases that
$q=3,5,7,11$ and $13$.

Throughout this paper, all modules and comodules are left modules
and left comodules, and moreover they are finite-dimensional over
$k$. $\otimes$, ${\rm dim}$ mean $\otimes _k$, ${\rm dim}_k$,
respectively. For two positive integers $m$ and $n$, $gcd(m,n)$
denotes the greatest common divisor of $m,n$. Our references for the
theory of Hopf algebras are \cite{Montgomery2} or \cite{Sweedler}.
The notation for Hopf algebras is standard. For example, the group
of group-like elements in $H$ is denoted by $G(H)$.

\section{Preliminaries}\label{sec2}
\subsection{Semisolvability}\label{sec2-1}
Let $H$ be a finite-dimensional Hopf algebra over $k$. A Hopf
subalgebra $A\subseteq H$ is called normal if $h_1AS(h_2)\subseteq
A$ and $S(h_1)Ah_2\subseteq A$, for all $h\in H$. If $H$ does not
contain proper normal Hopf subalgebras then it is called simple. The
notion of simplicity is self-dual, that is, $H$ is simple if and
only if $H^*$ is simple.

Let $q:H\to B$ be a Hopf algebra map and consider the subspaces of
coinvariants
$$H^{coq}=\{h\in H|(id\otimes q)\Delta(h)=h\otimes 1\}, \mbox{and\,}$$
$$^{coq}\!H=\{h\in H|(q\otimes id)\Delta(h)=1\otimes h\}.$$
Then $H^{coq}$ (respectively, $^{coq}H$) is a left (respectively,
right) coideal subalgebra of $H$. Moreover, we have
$${\rm dim}H ={\rm dim}H^{coq}{\rm dim}q(H) ={\rm dim}{}^{coq}H{\rm dim}q(H).$$

The left coideal subalgebra $H^{coq}$ is stable under the left
adjoint action of $H$. Moreover $H^{coq} ={}^{coq}H$ if and only if
$H^{coq}$ is a (normal) Hopf subalgebra of $H$. If this is the case,
we shall say that the map $q:H\to B$ is normal. See \cite{Schneider}
for more details.

The following lemma comes from \cite[Section 1.3]{Natale4}.
\begin{lem}\label{lem5}
Let $q:H\to B$ be a Hopf epimorphism and $A$ a Hopf subalgebra of
$H$ such that $A\subseteq H^{coq}$. Then ${\rm dim}A$ divides ${\rm
dim}H^{coq}$.
\end{lem}

The notions of upper and lower semisolvability for
finite-dimensional Hopf algebras have been introduced in
\cite{Montgomery}, as generalizations of the notion of solvability
for finite groups. By definition, $H$ is called lower semisolvable
if there exists a chain of Hopf subalgebras
$$H_{n+1} = k\subseteq
H_{n}\subseteq\cdots \subseteq H_1 = H$$ such that $H_{i+1}$ is a
normal Hopf subalgebra of $H_i$, for all $i$, and all quotients
$H_{i}/H_{i}H^+_{i+1}$ are trivial. Dually, $H$ is called upper
semisolvable if there exists a chain of quotient Hopf algebras
$$H_{(0)} =
H\xrightarrow{\pi_1}H_{(1)}\xrightarrow{\pi_2}\cdots\xrightarrow{\pi_n}H(n)
= k$$ such that each of the maps $H_{(i-1)}\xrightarrow{\pi_i}
H_{(i)}$ is normal, and all $H_{(i-1)}^{co\pi_i}$ are trivial.

By \cite[Corollary 3.3]{Montgomery}, we have that $H$ is upper
semisolvable if and only if $H^*$ is lower semisolvable. If this is
the case, then $H$ can be obtained from group algebras and their
duals by means of (a finite number of) extensions. For the
definition of the extension of Hopf algebras, the reader is directed
to \cite[Definition 1.3]{Masuoka3}.

Recall that a semisimple Hopf algebra $H$ is called of Frobenius
type if the dimensions of the simple $H$-modules divide the
dimension of $H$. Kaplansky conjectured that every
finite-dimensional semisimple Hopf algebra is of Frobenius type
\cite[Appendix 2]{Kaplansky}. It is still an open problem. Recently,
many examples show that a positive answer to Kaplansky's conjecture
would be very helpful in the classification problem. For example, in
case that ${\rm dim}H$ is a product of two distinct prime numbers,
Gelaki and Westreich \cite{Gelaki} proved that if $H$ and $H^*$ are
of Frobenius type then $H$ is trivial.

The following result is not explicitly stated in \cite{Etingof2}. We
give a proof for completeness.
\begin{lem}\label{lem1}
Let $H$ be a semisimple Hopf algebra of dimension $p^mq^n$, where
$p,q$ are distinct prime numbers and $m,n$ are non-negative integer.
Then $H$ is of Frobenius type and $H$ has a non-trivial
$1$-dimensional representation.
\end{lem}

The proof of Lemma \ref{lem1} involves some definitions and
properties from fusion categories. We refer the reader to
\cite{Etingof2} and references therein for basic results on fusion
category.

\begin{proof}
Let $Rep(H)$ be the category of representations of $H$. By
\cite[Theorem 1.6]{Etingof2}, $Rep(H)$ is a solvable fusion
category. Comparing \cite[Definition 1.1]{Etingof2} with
\cite[Definition 1.2]{Etingof2}, we find out that $Rep(H)$ is also
weakly group-theoretical. The first statement then follows from
\cite[Theorem 1.5]{Etingof2}. The second statement directly follows
from \cite[Proposition 9.9]{Etingof2}.
\end{proof}

\begin{lem}\label{lem2}
Let $H$ be a semisimple Hopf algebra of dimension $p^2q^2$, where
$p<q$ are prime numbers. If $H$ has a Hopf subalgebra $K$ of
dimension $pq^2$ then $H$ is lower semisolvable.
\end{lem}
\begin{proof}
Since the index of $K$ in $H$ is $p$ which is the smallest prime
number dividing ${\rm dim}H$, the result in \cite{Kobayashi} shows
that $K$ is a normal Hopf algebra of $H$. Since the dimension of the
quotient $H/HK^+$ is $p$, the result in \cite{Zhu} shows that it is
trivial.

Since $K^*$ is also a semisimple Hopf algebra (see \cite{Larson}),
Lemma \ref{lem1} and \cite[Theorem 5.4.1]{Natale3} show that $K$ has
a proper normal Hopf subalgebra $L$ of dimension $p,q,pq$ or $q^2$.
The results in \cite{Etingof,Gelaki,Masuoka,Masuoka2} (mentioned in
Section \ref{sec1}) show that $L$ and $K/KL^+$ are both trivial.
Hence, we have a chain of Hopf subalgebras $k\subseteq L\subseteq
K\subseteq H$, which satisfies the definition of lower
semisolvability.
\end{proof}

\subsection{Characters}\label{sec2-2}
Throughout this section, $H$ will be a semisimple Hopf algebra over
$k$.

Let $V$ be an $H$-module. The character of $V$ is the element
$\chi=\chi_V\in H^*$ defined by $\langle\chi,h\rangle={\rm Tr}_V(h)$
for all $h\in H$. The degree of $\chi$ is defined to be the integer
${\rm deg}\chi=\chi(1)={\rm dim}V$. We shall use $X_t$ to denote the
set of all irreducible characters of $H$ of degree $t$. If $U$ is
another $H$-module, we have $$\chi_{U\otimes
V}=\chi_U\chi_V,\quad\chi_{V^*}=S(\chi_V),$$ where $S$ is the
antipode of $H^*$.

Hence, the irreducible characters, namely, the characters of the
simple $H$-modules, span a subalgebra $R(H)$ of $H^*$, which is
called the character algebra of $H$. By \cite[Lemma 2]{Zhu}, $R(H)$
is semisimple. The antipode $S$ induces an anti-algebra involution
$*: R(H)\to R(H)$, given by $\chi\to\chi^*:=S(\chi)$. The character
of the trivial $H$-module is the counit $\varepsilon$.

The properties of $R(H)$ have been intensively studied in
\cite{Nichols}. We recall some of them here, and will use them
freely in this paper. See also \cite[Section 1.2]{Natale4}.

Let $\chi_U,\chi_V\in R(H)$ be the characters of the $H$-modules $U$
and $V$, respectively. The integer $m(\chi_U,\chi_V)={\rm
dimHom}_H(U,V)$ is defined to the the multiplicity of $U$ in $V$.
This can be extended to a bilinear form $m:R(H)\times R(H)\to k$.

Let $\widehat{H}$ denote the set of irreducible characters of $H$.
Then $\widehat{H}$ is a basis of $R(H)$. If $\chi\in R(H)$, we may
write $\chi=\sum_{\alpha\in \widehat{H}}m(\alpha,\chi)\alpha$. Let
$\chi,\psi,\omega\in R(H)$. Then
$m(\chi,\psi\omega)=m(\psi^*,\omega\chi^*)=m(\psi,\chi\omega^*)$ and
$m(\chi,\psi)=m(\chi^*,\psi^*)$. See \cite[Theorem 9]{Nichols}.

For each group-like element $g$ in  $G(H^*)$, we have
$m(g,\chi\psi)=1$, if $\psi=\chi^*g$ and $0$ otherwise for all
$\chi,\psi\in \widehat{H}$. In particular, $m(g,\chi\psi)=0$ if
$deg(\chi)\neq deg(\psi)$. Let $\chi\in \widehat{H}$. Then for any
group-like element $g$ in $G(H^*)$, $m(g,\chi\chi^{*})>0$ if and
only if $m(g,\chi\chi^{*})= 1$ if and only if $g\chi=\chi$. The set
of such group-like elements forms a subgroup of $G(H^*)$, of order
at most $({\rm deg}(\chi))^2$.  See \cite[Theorem 10]{Nichols}.
Denote this subgroup by $G[\chi]$. In particular, we have
$$\chi\chi^*=\sum_{g\in G[\chi]}g+\sum_{\alpha\in \widehat{H},{\rm
deg}\alpha>1}m(\alpha,\chi\chi^*)\alpha.$$

The following result can be found in \cite[Lemma 2.2.2]{Natale1}.
\begin{lem}\label{lem3}
Let $\chi\in \widehat{H}$ be an irreducible character of $H$. Then

 (1)\,The order of $G[\chi]$ divides $({\rm deg}\chi)^2$.

 (2)\,The order of $G(H^*)$ divides $n({\rm deg}\chi)^2$, where $n$ is the
 number of non-isomorphic simple $H$-modules of dimension ${\rm deg}\chi$.
\end{lem}

Let $1=d_1,d_2,\cdots,d_s$, $n_1,n_2,\cdots,n_s$ be positive
integers, with $d_1<d_2<\cdots<d_s$. $H$ is said to be of type
$(d_1,n_1;\cdots;d_s,n_s)$ as an algebra if $d_1,d_2,\cdots,d_s$ are
the dimensions of the simple $H$-modules and  $n_i$ is the number of
the non-isomorphic simple $H$-modules of dimension $d_i$. That is,
as an algebra,  $H$ is isomorphic to a direct product of full matrix
algebras $$H\cong k^{(n_1)}\times
\prod_{i=2}^{s}M_{d_i}(k)^{(n_i)}.$$

If $H^*$ is of type $(d_1,n_1;\cdots;d_s,n_s)$ as an algebra, then
$H$ is said to be of type $(d_1,n_1;\cdots;d_s,n_s)$ as a coalgebra.

A subalgebra $A$ of $R(H)$ is called a standard subalgebra if $A$ is
spanned by irreducible characters of $H$. Let $X$ be a subset of
$\widehat{H}$. Then $X$ spans a standard subalgebra of $R(H)$ if and
only if the product of characters in $X$ decomposes as a sum of
characters in $X$. There is a bijection between $*$-invariant
standard subalgebras of $R(H)$ and quotient Hopf algebras of $H$.
See \cite[Theorem 6]{Nichols}.

\begin{lem}\label{lem4}
Let $G$ be a non-trivial subgroup of $G(H^*)$. If $G[\chi_t]=G$ for
every $\chi_t\in X_t$, then $\chi_t\chi_t'$ is not irreducible for
all $\chi_t,\chi_t'\in X_t$.
\end{lem}

\begin{proof}
This is a consequence of \cite[Lemma 2.4.1]{Natale4}.
\end{proof}

\subsection{Radford's biproduct}\label{sec2-3} In what follows, we
briefly summarize results from \cite{Radford}. Let $A$ be a
semisimple Hopf algebra and let ${}^A_A\mathcal{YD}$ denote the
braided category of Yetter-Drinfeld modules over $A$. Let $R$ be a
semisimple Yetter-Drinfeld Hopf algebra in ${}^A_A\mathcal{YD}$.
Denote by $\rho :R\to A\otimes R$, $\rho (a)=a_{-1} \otimes a_0 $,
and $\cdot :A\otimes R\to R$, the coaction and action of $A$ on $R$,
respectively. We shall use the notation $\Delta (a)=a^1\otimes a^2$
and $S_R $ for the comultiplication and the antipode of $R$,
respectively.

Since $R$ is in particular a module algebra over $A$, we can form
the smash product (see \cite[Definition 4.1.3]{Montgomery}). This is
an algebra with underlying vector space $R\otimes A$, multiplication
is given by $$(a\otimes g)(b\otimes h)=a(g_1 \cdot b)\otimes g_2 h,
\mbox{\;for all\;}g,h\in A,a,b\in R,$$ and unit $1=1_R\otimes1_A$.

Since $R$ is also a comodule coalgebra over $A$, we can dually form
the smash coproduct. This is a coalgebra with underlying vector
space $R\otimes A$, comultiplication is given by $$\Delta (a\otimes
g)=a^1\otimes (a^2)_{-1} g_1 \otimes (a^2)_0 \otimes g_2
,\mbox{\;for all\;}h\in A,a\in R, $$ and counit
$\varepsilon_R\otimes\varepsilon_A$.

As observed by D. E. Radford (see \cite[Theorem 1]{Radford}), the
Yetter-Drinfeld condition assures that $R\otimes A$ becomes a Hopf
algebra with these structures. This Hopf algebra is called the
Radford's biproduct of $R$ and $A$. We denote this Hopf algebra by $
R\#A$ and write $a\# g=a\otimes g$ for all $g\in A,a\in R$. Its
antipode is given by
$$S(a\# g)=(1\# S(a_{-1} g))(S_R (a_0 )\# 1),\mbox{\;for
all\;}g\in A,a\in R.$$

A biproduct $R\#A$ as described above is characterized by the
following property(see \cite[Theorem 3]{Radford}): suppose that $H$
is a finite-dimensional Hopf algebra endowed with Hopf algebra maps
$\iota:A\to H$ and $\pi:H\to A$ such that $\pi \iota:A\to A$ is an
isomorphism. Then the subalgebra $R= H^{co\pi}$ has a natural
structure of Yetter-Drinfeld Hopf algebra over $A$ such that the
multiplication map $R\#A\to H$ induces an isomorphism of Hopf
algebras.

The following theorem is a direct consequence of \cite[Lemma
4.1.9]{Natale4}. We give the proof for the sake of completeness.
\begin{thm}\label{prop1}
Let $H$ be a semisimple Hopf algebra of dimension $p^2q^2$, where
$p,q$ are distinct prime numbers. If $gcd(|G(H)|,|G(H^*)|)=p^2$,
then $H\cong R\#kG$ is a biproduct, where $kG$ is the group algebra
of group $G$ of order $p^2$, $R$ is a semisimple Yetter-Drinfeld
Hopf algebra in $^{kG}_{kG}\mathcal{YD}$ of dimension $q^2$.
\end{thm}
\begin{proof}
By assumption and Sylow Theorem, $G(H^*)$ has a subgroup $K$ of
order $p^2$. Considering the Hopf algebra map $q: H\to (kK)^*$
obtained by transposing the inclusion $kK\subseteq H^*$, we have
that ${\rm dim}H^{coq}=q^2$. Again by assumption and Sylow Theorem,
$G(H)$ also has a subgroup $G$ of order $p^2$. If there exists an
element $1\neq g\in G$ such that $g$ appears in $H^{coq}$, then
$k\langle g\rangle\subseteq H^{coq}$ since $H^{coq}$ is a subalgebra
of $H$. But this contradicts Lemma \ref{lem5} since ${\rm
dim}k\langle g\rangle$ does not divide ${\rm dim}H^{coq}$.
Therefore, $H^{coq}\cap kG=k1$. This means that the restriction
$q|_{kG}$ is injective, and hence $q|_{kG}: kG\to (kK)^*$ is an
isomorphism. Finally, from the discussion above, we know that
$H\cong R\# kG$ is a biproduct, where $R=H^{coq}$.
\end{proof}

\section{Semisimple Hopf algebras of dimension $p^2q^2$}\label{sec3}
Let $p,q$ be distinct prime numbers with $p^4<q$, and $H$ a
semisimple Hopf algebra of dimension $p^2q^2$. By Nichols-Zoeller
Theorem \cite{Nichols2}, the order of $G(H^*)$ divides ${\rm dim}H$.
Moreover, $|G(H^*)|\neq1$ by Lemma \ref{lem1}. Again by Lemma
\ref{lem1}, the dimension of a simple $H$-module can only be
$1,p,p^2$ or $q$. Let $a,b,c$ be the number of non-isomorphic simple
$H$-modules of dimension $p,p^2$ and $q$, respectively. It follows
that we have an equation $p^2q^2=|G(H^*)|+ap^2+bp^4+cq^2$. In
particular, if $|G(H^*)|=p^2q^2$ then $H$ is a dual group algebra.

The proof of the following lemma is direct.
\begin{lem}\label{lem6}
The irreducible characters of degree $1, p$ and $p^2$ span a
standard subalgebra of $R(H)$ corresponding to a quotient Hopf
algebra $\overline{H}$ of $H$ of dimension $|G(H^*)|+ap^2+bp^4$. In
particular, $|G(H^*)|$ divides ${\rm dim}\overline{H}$ and
$|G(H^*)|+ap^2+bp^4$ divides ${\rm dim}H$.
\end{lem}

\begin{lem}\label{lem7}
If $|G(H^*)|=p$ or $pq$, then $H$ is upper semisolvable.
\end{lem}
\begin{proof}
First, $c\neq0$, since otherwise we get the contradiction $p^2\mid
p$.

Consider the quotient Hopf algebra $\overline{H}$ from Lemma
\ref{lem6}. Then $p\mid{\rm dim}\overline{H}$ and since $c\neq0$,
then ${\rm dim}\overline{H} < p^2q^2$. Therefore ${\rm
dim}\overline{H} = p, pq, p^2q, pq^2$ or $p^2$. Moreover, ${\rm
dim}\overline{H}\neq p^2$, since otherwise
$(\overline{H})^*\subseteq kG(H^*)$ by \cite{Masuoka2}, but $p^2
={\rm dim}\overline{H}$ does not divide $|G(H^*)|= p$ or $pq$.

The possibilities ${\rm dim}\overline{H}= p, pq$ or $p^2q$ lead,
respectively to the contradictions $p^2q^2= p+cq^2$,
$p^2q^2=pq+cq^2$ and $p^2q^2=p^2q+cq^2$. Hence these are also
discarded, and therefore ${\rm dim}\overline{H}=pq^2$. This implies
that $H$ is upper semisolvable, by Lemma \ref{lem2}.
\end{proof}

\begin{lem}\label{lem8}
$|G(H^*)|\neq q$.
\end{lem}
\begin{proof}
Suppose on the contrary that $|G(H^*)|= q$. By Lemma \ref{lem6},
${\rm dim}\overline{H}=q+ap^2+bp^4$. On the other hand, the product
of irreducible characters of $\overline{H}$ of degree $>1$ cannot
contain nontrivial characters of degree $1$, by Lemma \ref{lem3}(1).
If $a\neq0$ or $b \neq0$, this would imply $p^2=1+mp$ or $p^4=1+mp$
for some positive integer $m$, which is impossible. Therefore
$a=b=0$. So we have $p^2q^2=q+cq^2$, which is a contradiction.
\end{proof}

\begin{lem}\label{lem9}
If $|G(H^*)|=q^2$, then $H$ is upper semisolvable.
\end{lem}
\begin{proof}
A similar argument as in Lemma \ref{lem8} shows that $a=b=0$. Hence,
$H$ is of type $(1,q^2;q,p^2-1)$ as an algebra. Equivalently, $H^*$
is of type $(1,q^2;q,p^2-1)$ as a coalgebra. The group $G(H^*)$,
being abelian, acts by left multiplication on the set $X_q$. The set
$X_q$ is a union of orbits which have length $1,q$ or $q^2$. Since
$|X_q|=p^2-1$ is less than $q$, every orbit has length $1$. That is,
$G[\chi_q]=G(H^*)$ for all $\chi_q\in X_q$.  Let $g\in G(H^*)$ and
$\chi_q\in X_q$. Then $g\chi_q=\chi_q$ and
$g^{-1}\chi_q^*=\chi_q^*$. This means that $g\chi_q=\chi_qg=\chi_q$.

Let $C_i (i=1,\cdots,p^2-1)$ be the non-isomorphic $q^2$-dimensional
simple subcoalgebra of $H^*$. Then $gC_i=C_i=C_ig$ for all $g\in
G(H^*)$. By \cite[Proposition 3.2.6]{Natale4}, $G(H^*)$ is normal in
$k[C_1,\cdots,C_{p^2-1}]$, where $k[C_1,\cdots,C_{p^2-1}]$ denotes
the subalgebra generated by $C_1,\cdots,C_{p^2-1}$. It is a Hopf
subalgebra of $H^*$ containing $G(H^*)$.  Counting dimension, we
know $k[C_1,\cdots,C_{p^2-1}]=H^*$. Since $kG(H^*)$ is a group
algebra and the quotient $H^*/H^*(kG(H^*))^+$ is trivial (see
\cite{Masuoka2}), $H^*$ is lower semisolvable. Hence, $H$ is upper
semisolvable.
\end{proof}

From the discussion above, the following lemma is obvious.
\begin{lem}
If $|G(H^*)|=pq^2$, then $H$ is upper semisolvable.
\end{lem}

\begin{lem}\label{lem11}
If $|G(H^*)|=p^2$ or $p^2q$ then $H$ is either semisolvable or
isomorphic to a Radford's biproduct $R\# kG$, where $kG$ is the
group algebra of group $G$ of order $p^2$, $R$ is a semisimple
Yetter-Drinfeld Hopf algebra in ${}^{kG}_{kG}\mathcal{YD}$ of
dimension $q^2$.
\end{lem}
\begin{proof}
This is a corollary of Lemma \ref{lem2}.
\end{proof}

We are now in a position to give the main theorem.
\begin{thm}\label{thm1}
Let $H$ be a semisimple Hopf algebra of dimension $p^2q^2$, where
$p,q$ are prime numbers with $p^4<q$. Then $H$ is either
semisolvable or isomorphic to a Radford's biproduct $R\# kG$, where
$kG$ is the group algebra of group $G$ of order $p^2$, $R$ is a
semisimple Yetter-Drinfeld Hopf algebra in
${}^{kG}_{kG}\mathcal{YD}$ of dimension $q^2$.
\end{thm}

In analogy with the situations for finite groups, it is enough for
many applications to know that a Hopf algebra is semisolvable. Under
certain restrictions on $p$ and $q$, we can obtain a more precise
result.
\begin{cor}
If $p$ does not divide $q-1$ or $q+1$, then $H$ is semisolvable.
\end{cor}
\begin{proof}
It suffices to consider the case that the order of $G(H)$ and
$G(H^*)$ are $p^2$ or $p^2q$, and $H$ is a biproduct. Let $q: H\to
(kK)^*$ be the projection in Theorem \ref{prop1}. Then we have that
${\rm dim}H^{coq}=q^2$. We then consider the decomposition of
$H^{coq}$ as a coideal of $H$. Let $c$ be the number of
non-isomorphic irreducible left coideals of $H$ of dimension $q$. If
$|G(H)|=p^2$ then $c=0$, otherwise Lemma \ref{lem3} (2) shows that
$cq^2\geq p^2q^2$, a contradiction. If $|G(H)|=p^2q$ then $c=0$ by a
similar argument. Hence, by Lemma \ref{lem5}, there are $2$ possible
decompositions of $H^{coq}$ as a coideal of $H$:
$$H^{coq}=k1\oplus\sum_iV_i\oplus\sum_jW_j,\mbox{\,or\,}H^{coq}=kG\oplus\sum_iV_i\oplus\sum_jW_j,$$
where $V_i$ is an irreducible left coideal of $H$ of dimension $p$,
$W_i$ is an irreducible left coideal of $H$ of dimension $p^2$ and
$G$ is a subgroup of $G(H)$ of order $q$. Counting dimensions on
both sides, we have $q^2=1+mp$ or $q^2=q+np$ for some positive
integers $m,n$. This contradicts the assumption that $p$ does not
divide $q-1$ and $q+1$.
\end{proof}
 As an immediate consequence of the discussions in this section, we
 have the following corollary.
\begin{cor}
If $H$ is simple as a Hopf algebra then $H$ is isomorphic to a
Radford's biproduct $R\# kG$, where $kG$ is the group algebra of
group $G$ of order $p^2$, $R$ is a semisimple Yetter-Drinfeld Hopf
algebra in ${}^{kG}_{kG}\mathcal{YD}$ of dimension $q^2$.
\end{cor}

The following example was pointed out to the author by the anonymous
referee.

\begin{ex}
In fact, examples of nontrivial semisimple Hopf algebras of
dimension $p^2q^2$ which are Radford's biproducts in such a way, and
are simple as Hopf algebras do exists. A construction of such
examples as twisting deformations of certain groups appears in
\cite[Remark 4.6]{Galindo}.
\end{ex}

\section{Semisimple Hopf algebras of dimension $4q^2$}\label{sec4}
Let $q$ be a prime number, and $H$ a semisimple Hopf algebra of
dimension $4q^2$. In this section, we discuss the structure of $H$.
By Theorem \ref{thm1}, it suffices to consider the cases that
$q=3,5,7,11$ and $13$. By Nichols-Zoeller Theorem and Lemma
\ref{lem1}, the order of $G(H^*)$ is $2,4,q,q^2,2q,4q$,$2q^2$ or
$4q^2$. Moreover, if $|G(H^*)|=4q^2$ then $H$ is a dual group
algebra. The dimension of a simple $H$-module can only be $1,2,4$ or
$q$. Let $a,b,c$ be the number of non-isomorphic simple $H$-module
of dimension $2,4$ and $q$, respectively. Then we have
$4q^2=|G(H^*)|+4a+16b+cq^2$. In particular, if $c\neq 0$ then
$c=1,2$ or $3$. By \cite[Chapter 8]{Natale4}, if ${\rm dim}H=36$
then $H$ is upper semisolvable or lower semisolvable. Therefore, we
may assume that $q=5,7,11$ or $13$ in the followings.

\begin{lem}\label{lem12}
If $|G(H^*)|=2$ then $H$ is upper semisolvable.
\end{lem}
\begin{proof}
We first note that $c\neq 0$, otherwise $4q^2=2+4a+16b$ will give
rise to a contradiction $2(q^2-a-4b)=1$. That is, $c=1,2$ or $3$.

We then consider the case that $a\neq 0$. Let $\chi_2\in X_2$. Since
$H$ does not have irreducible characters of degree $3$, we have
$G[\chi_2]=G(H^*)$. Then a similar argument as in Lemma \ref{lem7}
shows that $G(H^*)\cup X_2$ spans a standard subalgebra of $R(H)$.
Hence, $H$ has a quotient Hopf algebra of dimension $2+4a$, and
$2+4a$ divides $4q^2$. Since $q$ is odd, $2+4a$ can not be $q$ and
$q^2$. If $2+4a=4q^2$ then $c=0$, a contradiction. If $2+4a=4q$ then
$1=2(q-a)$, a contradiction. If $2+4a=2q$, then a direct check, for
$q=5,7,11,13$ and $c=1,2,3$, shows that $b$ is not a integer, a
contradiction. Hence, $2+4a=2q^2$ and $H$ has a quotient Hopf
algebra of dimension $2q^2$. Therefore, $H$ is upper semisolvable by
Lemma \ref{lem2}.

Finally, we consider the case that $a=0$. In this case,
$4q^2=2+16b+cq^2$. A direct check, for $q=5,7,11,13$ and $c=1,2,3$,
shows that above equation holds true only when $b=3,q=5,c=2$ or
$b=6,q=7,c=2$ or $b=15,q=11,c=2$ or $b=21,q=13,c=2$. That is, $H$ is
of type $(1,2;4,3;5,2)$, $(1,2;4,6;7,2)$, $(1,2;4,15;11,2)$ or
$(1,2;4,21;13,2)$ as an algebra. We shall prove that all these can
not happen.

Suppose on the contrary that $H$ is of type $(1,2;4,3;5,2)$ as an
algebra. Let $\chi_4\in X_4$ and $G(H^*)=\{\varepsilon,g\}$. Then
there must exist $\chi_5\in X_5$ such that $1\leq
m(\chi_5,\chi_4\chi_4^*)\leq 3$.

If $m(\chi_5,\chi_4\chi_4^*)=3$ then $m(\chi_4,\chi_5\chi_4)=3$.
This implies that $\chi_5\chi_4=3\chi_4+\chi_4'+\chi_4''$, where
$\chi_4\neq \chi_4',\chi_4\neq\chi_4''\in X_4$. In case $\chi_4'\neq
\chi_4''$, we have
$m(\chi_4',\chi_5\chi_4)=m(\chi_5,\chi_4'\chi_4^*)=1$. This implies
that $\chi_4'\chi_4^*=\chi_5+\varphi$, where $m(\chi_5,\varphi)=0$
and $deg\varphi=11$. Since $\chi_4'\neq\chi_4$, $\varepsilon$ can
not appear in $\varphi$. From the introduction in Section
\ref{sec2-2}, we know that the multiplicity of $g$ in $\varphi$ is
less than $2$. Hence, $\varphi=2\chi_5'+g$, where $\chi_5\neq
\chi_5'\in X_5$. From $m(g,\chi_4'\chi_4^*)=m(\chi_4',g\chi_4)=1$,
we have $g\chi_4=\chi_4'$. Hence,
$\chi_4'\chi_4^*=g\chi_4\chi_4^*=g+\chi_5+2\chi_5'$. This means that
$\chi_4\chi_4^*=\varepsilon+g\chi_5+2g\chi_5'=\varepsilon+\chi_5'+2\chi_5$.
In the second equality, we use the fact that $g\chi_5=\chi_5'$ which
is deduced from the fact that
$G[\chi_5]=G[\chi_5']=\{\varepsilon\}$. This contradicts the
assumption that $m(\chi_5,\chi_4\chi_4^*)=3$. In case
$\chi_4'=\chi_4''$, we have
$m(\chi_4',\chi_5\chi_4)=m(\chi_5,\chi_4'\chi_4^*)=2$. This implies
that $\chi_4'\chi_4^*=2\chi_5+\chi_5'+g$. A similar argument shows
that it is also a contradiction.

If $m(\chi_5,\chi_4\chi_4^*)=2$ then $m(\chi_4,\chi_5\chi_4)=2$.
This implies that $\chi_5\chi_4=2\chi_4+2\chi_4'+\chi_4''$, where
$\chi_4\neq \chi_4',\chi_4\neq\chi_4''\in X_4$. In case $\chi_4'=
\chi_4''$, we have
$m(\chi_4',\chi_5\chi_4)=m(\chi_5,\chi_4'\chi_4^*)=3$. This implies
that $\chi_4'\chi_4^*=3\chi_5+g$. Then
$1=m(g,\chi_4'\chi_4^*)=m(\chi_4',g\chi_4)$ implies that
$g\chi_4=\chi_4'$. Hence,
$\chi_4'\chi_4^*=g\chi_4\chi_4^*=g+3\chi_5$. This means that
$\chi_4\chi_4^*=\varepsilon+3g\chi_5=\varepsilon+3\chi_5'$. This
contradicts the assumption that $m(\chi_5,\chi_4\chi_4^*)=2$. In
case $\chi_4'\neq \chi_4''$, we have
$m(\chi_4',\chi_5\chi_4)=m(\chi_5,\chi_4'\chi_4^*)=2$. This implies
that $\chi_4'\chi_4^*=2\chi_5+\chi_5'+g$, where $\chi_5\neq
\chi_5'\in X_5$. A similar argument shows that
$\chi_4\chi_4^*=\varepsilon+2\chi_5'+\chi_5$. This also contradicts
the assumption.

If $m(\chi_5,\chi_4\chi_4^*)=1$ then
$\chi_4\chi_4^*=\varepsilon+\chi_5+2\chi_5'$, where
$\chi_5\neq\chi_5'\in X_5$. In this case,
$m(\chi_5',\chi_4\chi_4^*)=2$. From the discussion above, we know it
is impossible.

Suppose on the contrary that $H$ is of type $(1,2;4,6;7,2)$ as an
algebra. Let $\chi_4\in X_4$ and $G(H^*)=\{\varepsilon,g\}$. Then
there must exist $\chi_7\in X_7$ such that $1\leq
m(\chi_7,\chi_4\chi_4^*)\leq 2$.

If $m(\chi_7,\chi_4\chi_4^*)=1$ then $m(\chi_4,\chi_7\chi_4)=1$.
This implies that $\chi_7\chi_4=\chi_4+\varphi$, where
$m(\chi_4,\varphi)=0$ and $deg\varphi=24$. A direct check shows that
there is no irreducible character of degree $7$ in $\varphi$ and
there exists $\chi_4\neq\chi_4'\in X_4$ such that
$m(\chi_4',\chi_7\chi_4)=2$. Then $m(\chi_7,\chi_4'\chi_4^*)=2$
implies that $\chi_4'\chi_4^*=2\chi_7+\psi$, where $deg\psi=2$.
Since $\chi_4\neq \chi_4'$, $\varepsilon$ does not appear in the
decomposition of $\psi$. Hence, $\psi$ is irreducible or a sum of
$2$ copies of $g$. It is impossible.

If $m(\chi_7,\chi_4\chi_4^*)=2$ then $m(\chi_4,\chi_7\chi_4)=2$.
This implies that $\chi_7\chi_4=2\chi_4+\varphi$, where
$m(\chi_4,\varphi)=0$ and $deg\varphi=20$. From  the discussion
above, we know there does not exist $\chi_4\neq\chi_4'\in X_4$ such
that $m(\chi_4',\chi_7\chi_4)=2$. Then we have
$\chi_7\chi_4=2\chi_4+\sum_{i=1}^5\varphi_i$, where
$\{\chi_4,\varphi_1,\cdots,\varphi_5\}=X_4$. From
$m(\varphi_i,\chi_7\chi_4)=m(\chi_7,\varphi_i\chi_4^*)=1$, we have
$\varphi_i\chi_4^*=\chi_7+\psi_i$, where $deg\psi_i=9$ and
$m(\chi_7,\psi_i)=0$. It is clear that $m(g,\psi_i)=1$ for all $i$.
Then $m(g,\varphi_i\chi_4^*)=m(\varphi_i,g\chi_4)$ implies that
$\varphi_i=g\chi_4$. This means that $\varphi_1=\cdots=\varphi_5$, a
contradiction.

Suppose on the contrary that $H$ is of type $(1,2;4,15;11,2)$ as an
algebra. Let $\chi_4\in X_4$ and $G(H^*)=\{\varepsilon,g\}$. Then
$\chi_4\chi_4^*=\varepsilon+\chi_{11}+\varphi_1$, where
$\chi_{11}\in X_{11}$ and $\varphi_1\in X_4$. From
$m(\chi_{11},\chi_4\chi_4^*)=m(\chi_4,\chi_{11}\chi_4)=1$, we have
$\chi_{11}\chi_4=\chi_4+\varphi$, where $m(\chi_4,\varphi)=0$ and
$deg\varphi=40$. A direct check shows that there exists
$\chi_4\neq\chi_4'\in X_4$ such that
$m(\chi_4',\chi_{11}\chi_4)=m(\chi_{11},\chi_4'\chi_4^*)=1$. This
means that $\chi_4'\chi_4^*=\chi_{11}+\varphi_2+g$, where
$\varphi_2\in X_4$. Then $m(g,\chi_4'\chi_4^*)=m(\chi_4',g\chi_4)$
implies that $\chi_4'=g\chi_4$. Hence,
$\chi_4'\chi_4^*=g\chi_4\chi_4^*=\chi_{11}+\varphi_2+g$ implies that
$\chi_4\chi_4^*=g\chi_{11}+g\varphi_2+\varepsilon$. On the other
hand, $\chi_4\chi_4^*=\varepsilon+\chi_{11}+\varphi_1$. Hence,
$\chi_{11}=g\chi_{11}$. This means that $g$ appears in the
decomposition of $\chi_{11}\chi_{11}^*$, and hence
$G[\chi_{11}]=G(H^*)$. This contradicts the fact that the order of
$G[\chi_{11}]$ divides $121$ (See Lemma \ref{lem3}).

Suppose on the contrary that $H$ is of type $(1,2;4,21;13,2)$ as an
algebra. Let $\chi_4\in X_4$. Then the decomposition of
$\chi_4\chi_4^*$ gives a contradiction.
\end{proof}

\begin{lem}\label{lem13}
$|G(H^*)|\neq q$.
\end{lem}
\begin{proof}
Suppose on the contrary that $|G(H^*)|= q$. If $a\neq 0$ we then
take $\chi_2\in X_2$. Since $G[\chi_2]$ is a subgroup of $G(H^*)$
and the order of $G[\chi_2]$ divides $4$ by Lemma \ref{lem3}(1),
$G[\chi_2]=\{\varepsilon\}$ is trivial. This is a contradiction
since $H$ does not have irreducible characters of degree $3$. Hence,
$a=0$ and $4q^2=q+16b+cq^2$. A direct check, for $q=5,7,11,13$ and
$c=0,1,2,3$, shows that above equation holds true only when
$b=22,q=11,c=1$. That is, $H$ is of type $(1,11;4,22;11,1)$ as an
algebra. We shall prove that it is impossible.

Suppose on the contrary that $H$ is of type $(1,11;4,22;11,1)$ as an
algebra. Let $\chi$ be the unique irreducible character of degree
$11$ and $g$ the generator of $G(H^*)$. Then $g\chi=\chi$ and hence
$G[\chi]=G(H^*)$. If $$\chi\chi^*=\chi^2=\sum_{i=1}^{11}g^i+10\chi$$
then $G(H^*)\cup X_{11}$ spans a standard subalgebra of $R(H)$.
Hence, $H$ has a quotient Hopf algebra of dimension $132$. By
Nichols-Zoeller Theorem, it is impossible.

Therefore, there exists $\chi_4\in X_4$ such that
$m(\chi_4,\chi^2)=n\geq1$. From
$m(\chi,\chi_4\chi)=m(\chi,\chi\chi_4^*)=n$, we have
$\chi\chi_4^*\stackrel{(1)}{=}n\chi+\varphi$, where
$deg\varphi=44-11n$ and $m(\chi,\varphi)=0$. On the other hand,
$\chi_4^*\chi_4=\varepsilon+\chi_4'+\chi$, where $\chi_4'\in X_4$.
From $m(\chi,\chi_4^*\chi_4)=m(\chi_4^*,\chi\chi_4^*)=1$, we have
$\chi\chi_4^*\stackrel{(2)}{=}\chi_4^*+\psi$, where $deg\psi=40$ and
$m(\chi_4^*,\psi)=0$. A direct check shows that $\chi$ does not
appear in the decomposition of $\psi$. Hence, $(1)$ and $(2)$ give
rise to a contradiction.
\end{proof}

\begin{lem}\label{lem14}
If $|G(H^*)|=q^2$ then $H$ is upper semisolvable.
\end{lem}
\begin{proof}
A similar argument as in Lemma \ref{lem13} shows that $a=0$, and
hence $3q^2=16b+cq^2$. A direct check, for $c=0,1,2,3$, shows that
$H$ is of type $(1,q^2;q,3)$ as an algebra. The result then follows
from a similar argument as in Lemma \ref{lem9}.
\end{proof}

\begin{lem}\label{lem14}
If $|G(H^*)|=2q$ then $H$ is upper semisolvable.
\end{lem}
\begin{proof}
If $a\neq 0$ then $q^2$ does not divide $a$, otherwise $4a\geq
4q^2$, a contradiction. Then, by Lemma \ref{lem6}, we have that
$2q+4a$ divides $4q^2$. A direct check shows that $2q+4a$ can not be
$q^2$, $4q^2$ and $4q$. Hence, $2q+4a=2q^2$ and $H$ has a quotient
Hopf algebra of dimension $2q^2$. So, $H$ is upper semisolvable by
Lemma \ref{lem2}.

If $a=0$ then $4q^2=2q+16b+cq^2$. A direct check, for $q=5,7,11,13$
and $c=0,1,2,3$, shows that this can not happen.
\end{proof}

The following lemma is obvious.
\begin{lem}
If $|G(H^*)|=2q^2$ then $H$ is upper semisolvable.
\end{lem}

\begin{lem}\label{lem15}
If $|G(H^*)|=4$ or $4q$ then $H$ is either semisolvable or
isomorphic to a Radford's biproduct $R\# kG$, where $kG$ is the
group algebra of group $G$ of order $4$, $R$ is a semisimple
Yetter-Drinfeld Hopf algebra in ${}^{kG}_{kG}\mathcal{YD}$ of
dimension $q^2$.
\end{lem}
\begin{proof}
This is a corollary of Lemma \ref{lem2}.
\end{proof}

Now we reach the main result in this section.
\begin{thm}\label{thm2}
Let $q$ be a prime number, and $H$ a semisimple Hopf algebra of
dimension $4q^2$. Then $H$ is either semisolvable or isomorphic to a
Radford's biproduct $R\# kG$, where $kG$ is the group algebra of
group $G$ of order $4$, $R$ is a semisimple Yetter-Drinfeld Hopf
algebra in ${}^{kG}_{kG}\mathcal{YD}$ of dimension $q^2$.
\end{thm}

\textbf{Acknowledgments:}\quad The author would like to thank the
referee for his/her valuable comments and suggestions, in particular
for providing a new proof of Lemma \ref{lem7}, which shortens the
original version. This work was partially supported by the FNS of
CHINA (10771183) and Docurate foundation of Ministry of Education of
CHINA (200811170001).


\begin{thebibliography}{00}
\bibitem{Andruskiewitsch}N. Andruskiewitsch, About finite dimensional Hopf algebras. Contemp.
Math. 294, 1--57 (2002).


\bibitem{Etingof}P. Etingof and S. Gelaki, Semisimple Hopf algebras of dimension $pq$ are
trivial. J. Algebra 210(2),  664--669 (1998).

\bibitem{Etingof2}P. Etingof, D. Nikshych and V. Ostrik, Weakly group-theoretical and solvable fusion
categories. Adv. Math. 226 (1), 176--505 (2011).

\bibitem{Galindo}C. Galindo, S. Natale, Simple Hopf algebras and deformations of
finite groups, Math. Res. Lett. 14 (6), 943--954 (2007).

\bibitem{Gelaki}S. Gelaki and S. Westreich, On semisimple Hopf algebras of dimension
$pq$. Proc. Amer. Math. Soc. 128(1), 39--47 (2000).

\bibitem{Kaplansky}I. Kaplansky, Bialgebras. Chicago, University of Chicago
Press 1975.

\bibitem{Kobayashi}T. Kobayashi and A. Masuoka, A result extended from groups to Hopf
algebras. Tsukuba J. Math. 21(1), 55-58 (1997).

\bibitem{Larson}R. G. Larson and D. E. Radford,  Finite dimensional cosemisimple Hopf
algebras in characteristic $0$ are semisimple. J. Algebra 117,
267--289(1988).

\bibitem{Masuoka}A. Masuoka, Semisimple Hopf algebras of dimension $2p$. Comm.
Algebra 23, 1931--1940 (1995).

\bibitem{Masuoka2}A. Masuoka, The $p^n$ theorem for semisimple Hopf
algebras. Proc. Amer. Math. Soc. 124, 735--737 (1996).

\bibitem{Masuoka3}A. Masuoka, Coideal subalgebras in finite Hopf algebras. J. Algebra 163, 819-831 (1994).


\bibitem{Montgomery}S. Montgomery and S. Whiterspoon, Irreducible representations of
crossed products. J. Pure Appl. Algebra 129, 315--326 (1998).


\bibitem{Montgomery2}S. Montgomery, Hopf algebras and their actions on rings. CBMS Reg.
Conf. Ser. Math. 82. Providence. Amer. Math. Soc. 1993.


\bibitem{Natale3}S. Natale, On semisimple Hopf algebras of dimension $pq^r$. Algebras Represent.
Theory 7 (2), 173--188 (2004).


\bibitem{Natale4}S. Natale, Semisolvability of semisimple Hopf algebras of low dimension. Mem. Amer. Math. Soc.
186 (2007).

\bibitem{Natale1}S. Natale, On semisimple Hopf algebras of dimension $pq^2$. J. Algebra
221(2), 242--278 (1999).

\bibitem{Nichols}W. D. Nichols and M. B. Richmond, The Grothendieck group of a Hopf
algebra. J. Pure Appl. Algebra 106, 297--306(1996).

\bibitem{Nichols2}W. D. Nichols and M. B. Zoelle, A Hopf algebra freeness theorem. Amer.
J. Math. 111(2), 381--385 (1989).


\bibitem{Radford}D. Radford, The structure of Hopf algebras with a
projection. J. Algebra 92, 322--347 (1985).


\bibitem{Schneider}H.-J. Schneider, Normal basis and
transitivity of crossed products for Hopf algebras. J. Algebra 152,
289--312 (1992).

\bibitem{Sweedler}M. E. Sweedler, Hopf Algebras. New
York, Benjamin 1969.


\bibitem{Zhu}Y. Zhu, Hopf algebras of prime dimension. Internat. Math. Res. Notices
1, 53--59 (1994).
\end{thebibliography}
\end{document}